\documentclass[11pt, reqno, psamsfonts]{amsart}
\pdfoutput=1

\usepackage{amssymb}
\usepackage{amsthm}
\usepackage{amsmath}
\usepackage{latexsym}
\usepackage[T1]{fontenc}
\usepackage[normalem]{ulem}
\usepackage[utf8]{inputenc}
\usepackage[russian, french, english]{babel}

\usepackage{graphicx}
\usepackage{wrapfig}
\usepackage{mathtools}
\usepackage{tikz}
\usepackage{amsbsy}
\usepackage[inline]{enumitem}
\usepackage{mathrsfs}
\usetikzlibrary{shapes,snakes}
\usetikzlibrary{arrows.meta}
\usetikzlibrary{decorations.pathmorphing}
\usetikzlibrary{patterns}
\usetikzlibrary{calc}
\usepackage{float}
\usepackage{array}
\usepackage{subcaption}
\usepackage{multicol}
\usepackage{stmaryrd}
\usepackage{cancel} 
\usepackage{algorithm}
\usepackage{comment}
\usepackage{algpseudocode}
\usepackage{xcolor}

\makeatletter
\algnewcommand{\LineComment}[1]{\Statex \hskip\ALG@thistlm {\color{gray}\texttt{// #1}}}
\makeatother

\counterwithin{algorithm}{section}
\usepackage[justification=centering, labelfont=bf]{caption}

\usepackage{lmodern}
\usepackage{mathabx}

\usepackage{adjustbox}
\usepackage{upgreek}
\usepackage{titlesec}
\usepackage{bbm}
\usepackage[spacing=true,kerning=true,babel=true,tracking=true]{microtype}
\usepackage[shortcuts]{extdash}
\usepackage[foot]{amsaddr}
\usepackage[left=1in,right=1in,top=1in,bottom=1in,bindingoffset=0cm]{geometry}
\usepackage{bm}
\usepackage{centernot}
\usepackage{mdframed}
\usepackage{hyperref}
\usepackage{mathdots}
\usepackage{xspace}
\definecolor{OliveGreen}{HTML}{3C8031}
\hypersetup{
    colorlinks=true,
    linkcolor=blue,
    filecolor=magenta,      
    urlcolor=red,
    citecolor=OliveGreen,
}
\allowdisplaybreaks

\usepackage[backend=biber, style=alphabetic, sorting=nyt, maxnames=100,backref=true, firstinits=true, sortcites = true]{biblatex}
\title{Choosability of multipartite hypergraphs}
\date{}

\author{\lsstyle Peter~Bradshaw}
\email{pb38@illinois.edu}

\author{\lsstyle Abhishek~Dhawan}
\email{adhawan2@illinois.edu}

\author{\lsstyle Nhi~Dinh}
\email{nhidinh2@illinois.edu}

\author{\lsstyle Shlok~Mulye}
\email{mulye2@illinois.edu}

\author{\lsstyle Rohan~Rathi}
\email{rarathi2@illinois.edu}

\address{\normalfont{}\textls{Department of Mathematics, University of Illinois Urbana-Champaign, Urbana, IL, USA}}
\thanks{Peter Bradshaw and Abhishek Dhawan received funding from NSF RTG grant DMS-1937241.}

\newtheoremstyle{bfnote}%
{}{}%
{\slshape}{}%
{\bfseries}{\bfseries.}%
{ }%
{\thmname{#1}\thmnumber{ #2}\thmnote{ \ep{\normalfont{}#3}}}

\theoremstyle{bfnote}
\newtheorem{theorem}{Theorem}[section]
\newtheorem*{theorem*}{Theorem}
\newtheorem{prop}[theorem]{Proposition}
\newtheorem*{prop*}{Proposition}
\newtheorem{claim}[theorem]{Claim}

\newtheorem{conj}[theorem]{Conjecture}

\newtheorem*{corl*}{Corollary}

\theoremstyle{definition}
\newtheorem{defn}[theorem]{Definition}
\newtheorem*{defn*}{Definition}

\newtheorem*{exmp*}{Example}

\theoremstyle{remark}
\newtheorem*{ques*}{Question}
\newtheorem*{remk*}{Remark}

\newcommand*{\myproofname}{Proof of claim}
\newenvironment{claimproof}[1][\myproofname]{\begin{proof}[#1]}{\end{proof}}

\newcommand{\N}{{\mathbb{N}}}

\renewcommand{\P}{\mathbb{P}}

\renewcommand{\epsilon}{\varepsilon}
\newcommand{\eps}{\epsilon}
\renewcommand{\phi}{\varphi}
\renewcommand{\theta}{\vartheta}
\renewcommand{\leq}{\leqslant}
\renewcommand{\geq}{\geqslant}
\newcommand{\defeq}{\coloneqq}

\newcommand{\bemph}[1]{{\normalfont#1}} 
\newcommand{\ep}[1]{\bemph{(}#1\bemph{)}} 

\newcommand{\emphdef}[1]{\textbf{\textit{{#1}}}}

\newcommand{\emphd}[1]{\emphdef{#1}}

\newcommand{\LLL}{\text{Lov\'asz Local Lemma}\xspace}

\titleformat{\subsection}[block]{\bfseries}{\thesubsection.}{1ex}{}
\titleformat{\subsubsection}[runin]{\itshape}{\bfseries\upshape\thesubsubsection.}{1ex}{}[.---]

\titleformat{\section}[block]{\scshape\filcenter}{\thesection.}{1ex}{}

\titlespacing*{\section}{0pt}{*3}{*1}
\titlespacing*{\subsection}{0pt}{*3}{*1}
\titlespacing*{\subsubsection}{0pt}{*1.5}{*0}

\renewbibmacro{in:}{}
\renewbibmacro*{volume+number+eid}{%
	\printfield{volume}%
	\setunit*{\addnbspace}
	\printfield{number}%
	\setunit{\addcomma\space}%
	\printfield{eid}}

\DeclareFieldFormat[article]{volume}{\textbf{#1}\space}
\DeclareFieldFormat[article]{number}{\mkbibparens{#1}}

\DeclareFieldFormat{journaltitle}{#1,}
\DeclareFieldFormat[thesis]{title}{\mkbibemph{#1}\addperiod}
\DeclareFieldFormat[article, unpublished, thesis]{title}{\mkbibemph{#1},}
\DeclareFieldFormat[book]{title}{\mkbibemph{#1}\addperiod}
\DeclareFieldFormat[unpublished]{howpublished}{#1, }

\DeclareFieldFormat{pages}{#1}

\DeclareFieldFormat[article]{series}{Ser.~#1\addcomma}

\setlist{topsep=3pt,itemsep=3pt}

\begin{document}

\sloppy

\vspace*{0pt}

\maketitle

\begin{abstract}
    A $k$-uniform hypergraph (or $k$-graph) $H = (V, E)$ is $k$-partite if $V$ can be partitioned into $k$ sets $V_1, \ldots, V_k$ such that each edge in $E$ contains precisely one vertex from each $V_i$. We show that $k$-partite $k$-graphs of maximum degree $\Delta$ are $q$-choosable for  $q \geq \left(\frac{4}{5}(k-1 + o(1))\Delta/\log \Delta\right)^{1/(k-1)}$. Our proof yields an efficient randomized algorithm for finding such a coloring, which shows that the conjectured algorithmic barrier for coloring pseudorandom $k$-graphs does not apply to $k$-partite $k$-graphs.
\end{abstract}

\section{Introduction}\label{section: intro}

\subsection{Background and results}\label{subsection: background}

All hypergraphs considered in this paper are finite and undirected. 
A \emphd{$k$-uniform hypergraph} (or \emphd{$k$-graph}) is an ordered pair $(V, E)$, where $V$ is a set of vertices and $E$ is a collection of $k$-element subsets of $V$ called edges.
For each $v \in V$, we let $E_H(v)$ denote the edges containing $v$, $N_H(v)$ denote the set of vertices contained in the edges in $E_H(v)$ apart from $v$ itself, $\deg_H(v) \defeq |E_H(v)|$, and $\Delta(H) \defeq \max_{u\in V}\deg_H(u)$.
We drop the subscript when $H$ is clear from context.
In this paper, we study multipartite hypergraphs, a natural extension of bipartite graphs to the hypergraph setting.

\begin{defn}\label{def: multipartite}
    A $k$-uniform hypergraph $H = (V, E)$ is \emphd{$k$-partite} if there exists a partition $V \defeq V_1\sqcup \cdots \sqcup V_k$ such that each edge $e \in E$ contains precisely one vertex from each set $V_i$.
\end{defn}

Multipartite hypergraphs find a wide array of applications in satisfiability problems, Steiner triple systems, and particle tracking in physics, to name a few.
In this work, we investigate the behavior of the choice number of such hypergraphs.
Before we state our main result, we establish
a few definitions.
An \emph{independent set} in 
a hypergraph
$H = (V, E)$ is a set $J \subseteq V$ that spans no edges, i.e., $e \not \subseteq J$ for each $e \in E$.
A \emph{proper coloring} of $H$ is an assignment of integers (referred to as colors) 
to the vertices of $H$ such that each color class (a maximal set of vertices assigned the same color) is an independent set.
The \emphd{chromatic number} of $H$, denoted by $\chi(H)$, is the minimum number of colors required for a proper coloring of $H$.
A \emph{list assignment} for $H$ is a function $L\,:\,V(H) \to 2^\N$\footnote{For a set $X$, we let $2^X$ denote its power set.} that assigns a set $L(v)$ of colors to each vertex $v \in V(H)$, represented by a set of nonnegative integers.
Given a list assignment $L$ for $H$, a \emph{proper $L$-coloring} of $H$ is a proper coloring in which
each vertex $v$ receives a color from its list $L(v)$.
For $q\in \N$, we say that $H$ is \emphd{$q$-choosable} if $H$ admits a proper $L$-coloring with respect to every list assignment $L$ satisfying $|L(v)| \geq q$ for each $v$.
The \emphd{choice number} of $H$, denoted by $\ch(H)$, is the minimum value $q$ such that $H$ is $q$-choosable.

When considering a $k$-partite $k$-graph $H$ with a vertex partition $V(H) = V_1 \sqcup \cdots \sqcup V_k$, $\bigcup_{i \in J}V_i$ is an independent set for every $J \subsetneq [k]$, and a trivial $2$-coloring of $H$ exists; namely, color the vertices in $V_1$ with color $1$, and color all other vertices with color $2$.
The choice number of $H$, however, can be much larger than $2$.
For $k = 2$, Erd\H{o}s, Rubin, and Taylor showed that $\ch(K_{n, n}) = (1+o(1))\log_2n$\footnote{Throughout this work, we use the standard asymptotic notation $O(\cdot)$, $\Omega(\cdot)$, $o(\cdot)$, etc.}
in the same paper that they introduced the choice number \cite{erdos1979choosability}.
Alon and Krivelevich also showed that Erd\H{o}s--R\'enyi random bipartite graphs satisfy $\ch(G_{\mathrm{bip}}(n, p)) = (1\pm o(1))\log_2(np)$ with high probability for a wide range of $p$ \cite{alon1998choice}.
This led them to pose the following conjecture:

\begin{conj}[{\cite[Conjecture~5.1]{alon1998choice}}]\label{conj:ak}
    For any bipartite graph $G$ of maximum degree at most $\Delta$, we have $\ch(G) = O(\log \Delta)$.
\end{conj}

While the conjecture is still open, there has been progress in recent years.
Alon, Cambie, and Kang showed that $\ch(G) \leq (1+o(1))\Delta/\log \Delta$ in \cite{alon2021asymmetric} (see also \cite{cambie2022independent} for an extension of the result to correspondence coloring), which matches the so-called \textit{shattering threshold} in coloring random graphs (we discuss this further in \S\ref{subsection: prior work}).
Very recently, Mohar, Stacho, and the first named author of this manuscript improved upon the above to show the following:

\begin{theorem}[{\cite[Theorem~1.2]{bradshaw2024bipartite}}]\label{theo: bipartite 4/5}
   The following holds for $\Delta$ sufficiently large.
   Let $G$ be a bipartite graph of maximum degree at most $\Delta$. 
   Then,
   \[\ch(G) < \frac{4}{5} \cdot \frac{\Delta}{\log \Delta}.\]
\end{theorem}

For $k \geq 3$, Haxell and Verstraete showed that the complete $k$-partite $k$-graph with $n$ vertices in each part (denoted by $K_{k*n}$) satisfies $\ch(K_{k*n}) \leq (1+o_n(1))\log_kn$ \cite{haxell2010list}, extending the aforementioned result of Erd\H{o}s, Rubin, and Taylor on bipartite graphs to the hypergraph setting.
In a similar flavor, M{\'e}roueh and Thomason generalized the result of Alon and Krivelevich to Erd\H{o}s--R\'enyi random $k$-partite $k$-graphs $\mathcal{H}(k, n, p)$ \cite{meroueh2019list}.
Here, each part $V_i$ contains $n$ vertices, and each potential edge is included independently with probability $p$ (a potential edge is a set consisting of exactly one vertex from each $V_i$).
They showed that $\ch(\mathcal{H}(k, n, p)) = \Theta_{k}\left(\log(n^{k-1}p)\right)$ almost surely for a wide range of $p$.
Recently, the second named author of this manuscript extended the result of Alon, Cambie, and Kang on bipartite graphs to all uniformities:

\begin{theorem}[{\cite[Theorem~3]{dhawan2025list}}]\label{theo: old}
    For all $\eps > 0$ and $k \geq 2$, the following holds for $\Delta$ sufficiently large.
    Let $H$ be a $k$-uniform $k$-partite hypergraph of maximum degree at most $\Delta$. 
    Then,
    \[\ch(H)\leq \left((k-1 + \eps)\frac{\Delta}{\log \Delta}\right)^{1/(k-1)}.\]
\end{theorem}

The above result matches the conjectured computational threshold for coloring random hypergraphs (we discuss this further in \S\ref{subsection: prior work}).
In the same paper, the author posed the following stronger form of Conjecture~\ref{conj:ak}:

\begin{conj}[{\cite[Conjecture~11]{dhawan2025list}}]\label{conj: k unif}
    For all $k \geq 2$, there is a constant $c \defeq c(k) > 0$ such that the following holds for $\Delta$ large enough.
    Let $H$ be a $k$-partite $k$-graph of maximum degree at most $\Delta$.
    Then, we have $\ch(H) \leq c\log \Delta$.
\end{conj}

Our main result improves the constant factor in Theorem~\ref{theo: old}, providing an advancement toward Conjecture~\ref{conj: k unif}.

\begin{theorem}\label{theorem: constant inside}
    For all $\eps > 0$ and $k \geq 2$, the following holds for $\Delta$ sufficiently large.
    Let $H$ be a $k$-uniform $k$-partite hypergraph of maximum degree at most $\Delta$. 
    Then,
    \[\ch(H)\leq \left(\frac{4}{5}(k-1 + \eps)\frac{\Delta}{\log \Delta}\right)^{1/(k-1)}.\]
\end{theorem}

For the case $k = 2$, this result was previously established in \cite[Theorem~2.4]{bradshaw2024bipartite}.
Indeed, our proof is inspired by their approach (see \S\ref{subsection: proof overview} for an overview).
The main result of~\cite{bradshaw2024bipartite} (see Theorem~\ref{theo: bipartite 4/5}) further improves the constant by optimizing a threshold function appearing in the proof.
We do not pursue such refinements here, instead focusing on techniques that extend cleanly to the general $k$-uniform setting.
We note that such refinements would yield only marginal gains in the constant factor of Theorem~\ref{theorem: constant inside}, and would not bridge the gap to the conjectured bound (Conjecture~\ref{conj: k unif}).

\subsection{Relation to prior work}\label{subsection: prior work}

In this section, we discuss related works
on coloring hypergraphs of bounded degree
to better place our result in context.

When considering graphs ($k = 2$), Brooks provided the first improvement upon the greedy bound \cite{brooks1941colouring};
he showed that $\chi(G) \leq \Delta(G)$ unless $G$ is complete or an odd cycle (in which case $\chi(G) = \Delta(G) + 1$).
Reed improved upon this showing that $\chi(G) \leq \Delta(G) - 1$ for sufficiently large $\Delta(G)$, provided $G$ does not contain a clique of size $\Delta(G)$ \cite{reed1999strengthening}.
Johansson provided an asymptotic improvement for $K_3$-free graphs of maximum degree $\Delta$, showing $\chi(G) = O(\Delta/\log\Delta)$ \cite{Joh_triangle}.
Recently, Molloy showed that the leading constant in the $O(\cdot)$ is at most $1 + o(1)$ (see also \cite{bernshteyn2019johansson, martinsson2021simplified, hurley2021first}), which is optimal up to a factor of $2$ \cite{BollobasIndependence}.
This coincides with the so-called shattering threshold for colorings of random graphs of average degree $\Delta$ \cite{Zdeborova,Achlioptas}.
Surpassing the bound of $\Delta/\log \Delta$ for any pseudorandom graph class is a tantalizing open problem in graph coloring.
A celebrated conjecture of Karp in average-case complexity theory \cite{karp1976probabilistic}, if true, implies that any proof technique that yields an efficient algorithm for pseudorandom graph coloring
cannot beat this bound, and so one expects all current approaches to fail to do so.
For a selection of works pertaining to coloring pseudorandom graphs, we direct the reader to \cite{AKSConjecture, DKPS, anderson2025coloring, anderson2023colouring, bradshaw2025toward, bonamy2022bounding}.

For $k > 2$, a simple \LLL argument shows that $\chi(H) = O\left(\Delta(H)^{1/(k-1)}\right)$ for all $k$-uniform hypergraphs $H$.
A hypergraph variant of Karp's conjecture appeared in recent work of Wang and the second named author of this manuscript \cite{dhawan2024low}.
In particular, the conjecture posits a computational threshold of $\left((k-1)\Delta/\log \Delta\right)^{1/(k-1)}$ on the number of colors
for efficiently
coloring pseudorandom hypergraphs.
While it is not known if any such class of hypergraphs meets this threshold, there are a number of results that match the growth rate.
Frieze and Mubayi \cite{frieze2013coloring} showed that there exists $c \coloneqq c(k) > 0$ such that $\chi(H) \leq c\left(\Delta/\log \Delta\right)^{1/(k-1)}$ for \textit{linear hypergraphs} (a hypergraph is linear if every pair of vertices is contained in at most one edge).
For $k = 3$, Cooper and Mubayi generalized this result to triangle-free hypergraphs \cite{cooper2016coloring} (a triangle in a hypergraph is a set of three pairwise intersecting edges with no common vertex).
Recently, Li and Postle extended this result to all uniformities \cite{li2022chromatic}.
The best-known bound on $\chi(H)$ for any pseudorandom hypergraph class is due to Iliopoulos \cite{iliopoulos2021improved}.
He showed that hypergraphs of girth at least five\footnote{Iliopoulos refers to such hypergraphs as \textit{locally sparse}. This terminology was also used in recent work of Verstraete and Wilson \cite{verstraete2024independent}, while the classical work of Ajtai, Koml\'os, Pintz, Spencer, and Szemer\'edi call them \textit{uncrowded} \cite{ajtai1982extremal}.}
have chromatic number at most $(k-1 + o(1))\left(\Delta/\log \Delta\right)^{1/(k-1)}$.

\subsection{Proof overview}\label{subsection: proof overview}
Our proof follows a random procedure combining the ideas of \cite{bradshaw2024bipartite} and \cite{dhawan2025list}.
We consider a $\Delta$-regular $k$-uniform $k$-partite hypergraph
$H$
with parts $V_1, \dots, V_k$ in which each vertex $v \in V(H)$ has a list $L(v)$ of $q$
colors.
We randomly assign
a color $\phi(u)$ 
to each vertex $u \in V_1 \cup \cdots \cup V_{k-1}$,
and then for each $v \in V_k$, we estimate the probability that some color $c \in L(v)$ can be 
assigned to $v$ without creating a monochromatic edge.
For each vertex $v \in V_k$, we 
aim to show that the probability that 
$\phi$ cannot be extended to $v$ without creating a monochromatic edge in $H$
is sufficiently small, so that we can apply the Lov\'asz Local Lemma and argue that with positive probability, $\phi$ can be extended to every $v \in V_k$.

The novel ingredient in our approach is the probability distribution used to assign a color $\phi(u)$ to each $u \in V_1 \cup \cdots \cup V_{k-1}$.
If we assign each color $\phi(u)$ uniformly at random from $L(u)$,
then the probability that $\phi$ cannot be extended to a given $v \in V_k$ is only sufficiently small when $q \geq \left ( (k-1 + o(1)) \frac{\Delta}{\log \Delta} \right )^{1/(k-1)}$, which is the main result of \cite{dhawan2025list}.
However, we find that when $\phi(u)$ is chosen
according to a certain non-uniform distribution
for each $u \in V_1$ and uniformly for each $u \in V_2 \cup \cdots \cup V_{k-1}$,
the probability that $\phi$ cannot be extended to a given $v\in V_k$ is sufficiently small when 
$q \geq \left ( \frac 45 (k-1 + o(1) ) \frac{\Delta}{\log \Delta} \right )^{1/(k-1)}$.
The idea of using a non-uniform distribution
to create a partial coloring was used by Mohar, Stacho, and the first named author of this manuscript \cite{bradshaw2024bipartite}
for list coloring bipartite graphs (Theorem~\ref{theo: bipartite 4/5}).
We show that for each $v \in V_k$,
with probability $1-\Delta^{-\omega(1)}$,
there exists a color $c \in L(v)$ that
can be assigned to $v$ without creating a monochromatic edge.
Thus, our approach succeeds with positive probability.

We note that it is unclear whether our approach can be optimized by choosing an appropriate non-uniform distribution that assigns colors to vertices in $V_2 \cup \cdots \cup V_{k-1}$ as well.
One obstacle of exploring more intricate approaches with multiple distributions
is that probabilities of certain events become difficult to analyze as our distributions become more involved.

\section{Preliminaries}\label{section: prelim}

In this section, we describe probabilistic tools that will be used to prove Theorem~\ref{theorem: constant inside}.
We start with the symmetric version of the Lov\'asz Local Lemma.

\begin{theorem}[{Lov\'asz Local Lemma; \cite[Corollary~5.1.2]{AlonSpencer}}]\label{theorem: Lovasz Local Lemma}
    Let $A_1$, $A_2$, \ldots, $A_n$ be events in a probability space. Suppose there exists $p_{LLL} \in [0, 1)$ such that for all $1 \leq i \leq n$ we have $\P[A_i] \leq p_{LLL}$. Further suppose that each $A_i$ is mutually independent from all but at most $d_{LLL}$ other events $A_j$ ($j\neq i$) for some $d_{LLL} \in \N$. If $ep_{LLL}(d_{LLL}+1) \leq 1$, then with positive probability none of the events $A_1$, \ldots, $A_n$ occurs.
\end{theorem}

We will also need the following special case of the FKG inequality, dating back to Harris \cite{Harris} and Kleitman \cite{Kleitman}.
Given a ground set $\Gamma$, a family $\mathcal A$ of subsets of $\Gamma$ is \emph{decreasing} if for each set $A \in \mathcal A$ and subset $A' \subseteq A$, we have $A' \in \mathcal A$ as well.
The original theorem is stated with regards to two decreasing families; however, as the intersection of decreasing families is decreasing, it can be shown that the inequality holds in the following more general form.

\begin{theorem}[{Harris's Inequality/Kleitman's Lemma \cite[Theorem 6.3.2]{AlonSpencer}}]\label{theo:harris}
    Let $\Gamma$ be a finite set, and let $S \subseteq \Gamma$ be a random subset of $\Gamma$ obtained by selecting each $x \in X$ independently with probability $p_x \in [0,1]$. If $\mathcal{A}_1, \ldots, \mathcal{A}_n$ are decreasing families of subsets of $\Gamma$, then 
    \[\P\left[S \in \bigcap_{i \in [n]}\mathcal{A}_i\right] \,\geq\, \prod_{i \in [n]}\P[S \in \mathcal{A}_i].\]
\end{theorem}

Finally, we use the following well-known corollary of Jensen's inequality:

\begin{theorem}[{Jensen's Inequality \cite[(3.6.1)]{hardy1952inequalities}}]\label{theo:Jensen}
Let $I\subseteq\mathbb{R}$ be an interval and let $f:I\to\mathbb{R}$ be convex. 
For any $x_1,\dots,x_t\in I$,
\[
\sum_{i=1}^t f(x_i)\;\ge\; t\,f\!\left(\frac{1}{t}\sum_{i=1}^t x_i\right).
\]
\end{theorem}

\section{Proof of Theorem~\ref{theorem: constant inside}}

Fix an integer $k \geq 2$ and let $\varepsilon>0$ be arbitrary. Let $\Delta \in \N$ be sufficiently large in terms of $\varepsilon$ and $k$. Without loss of generality, we may assume $\varepsilon < 1$ for the remainder of the proof. Let $H=(V,E)$ be a $k$-uniform $k$-partite hypergraph of maximum degree at most $\Delta$. Without loss of generality, we may assume that $H$ is $\Delta$-regular as, by standard arguments, any $k$-uniform $k$-partite hypergraph with maximum degree at most $\Delta$ can be embedded into a $\Delta$-regular $k$-partite $k$-graph.

We assume that each vertex $v\in V$ is assigned a list $L(v)$ of available colors. 
Without loss of generality, we may assume that all lists have the same size by arbitrarily deleting surplus colors from $L(v)$ if needed, that 
is, for some $q \in \mathbb N$, $|L(v)| = q$ for every $v \in V$.
We represent colors by nonnegative integers and use the usual order on $\mathbb{N}$.
Each list $L(v)$ inherits this order. 
For any ordered list $L \subseteq \N$ and any $c\in L$, define the index
\(I(L,c)=i\) if and only if $i-1$ elements of $L$ appear before $c$, i.e., $c$ is the $i$-th element of $L$ when considering the list in ascending order.

We will begin by randomly generating a partial coloring $\phi$ of $H$ which assigns colors to all vertices in $V_1 \cup \dots \cup V_{k-1}$, that is, the first $k-1$ vertex parts of $H$. We will then demonstrate that this coloring can be extended to the entire hypergraph $H$ with positive probability.

For each $u\in V_1\cup\cdots\cup V_{k-1}$, $\phi(u)$ is assigned the color $c$ chosen independently and at random from $L(u)$ according to the following probability distribution:
\begin{align}\label{eq: prob dist}
    \P_{u,c} \coloneqq \P[\phi(u) = c] = \begin{cases} \frac{\frac{8}{5}}{\,q\left(1 - \frac{3}{5q}\right)} \left(1 - \frac{3}{4q} \times I(L(u), c)\right) & \text{if } u \in V_{1};\\[12pt]
    \frac{1}{q} & \text{if } u\in V_2\cup\cdots \cup V_{k-1}.
\end{cases}
\end{align}
Observe that for each $u \in V_1 \cup \cdots \cup V_{k-1}$, we have $\sum_{c \in L(u)} \mathbb P_{u,c} = 1$. In particular, the above assignment defines a valid probability distribution. For each $e \in E$, define the \emphd{edge status}
\begin{align}\label{eq:  edge status}
\phi_e \coloneqq
\begin{cases}
c & \text{if } \phi(u)=c \text{ for all } u \in e \cap (V_1 \cup \cdots \cup V_{k-1});\\[2pt]
* & \text{otherwise.}
\end{cases}
\end{align}

We say that $e$ is \emphd{problematic with respect to color $c$} if $\phi_e = c$, and we define the probability of this event as
\[
\P_{e}(c) \coloneqq \P[\phi_e = c].
\]
Let 
\begin{align}\label{eq:  p definition}
p \coloneqq \frac{\frac{8}{5}}{\,q^{k-1}\left(1 - \frac{3}{5q}\right)} \left(1- \frac{3}{4q}\right).
\end{align}
Note that for each edge $e$ and color $c$,
\[
0 \le \P_{e}(c) \le p.
\]

The expected number of edges incident to a vertex \( v \in V_k \) that are problematic with color \( c \in L(v)\) is given by
\begin{align}\label{eq: rho}
    \rho(v,c) = \sum_{e \in E(v)} \P_e(c).
\end{align}
We say that \emphd{$c$ is blocked at $v$} if there exists an edge $e \in E(v)$ such that $e$ is problematic with respect to $c$. We also say that $v \in V_k$ is \emphd{blocked} if every color in $L(v)$ is blocked at $v$. 
Define the event
\begin{align}\label{Bad event definiton}
  \mathcal{B}_v \;\coloneqq\; 
  \left\{\,\forall\,c\in L(v) \text{, } \exists\,e\in E(v)\text{ such that } \phi_e = c\right\},
\end{align}
i.e., $\mathcal{B}_v$ denotes the event that $v$ is blocked.
The goal is to show that, with positive probability, none of the events $\{\mathcal{B}_v\}_{v \in V_k}$ occur, which implies that there exists an outcome of $\phi$ with no blocked vertex in $V_k$, and we may extend $\phi$ to all of $V$ by coloring each $v \in V_k$ with some color $c \in L(v)$ that is not blocked at $v$.

Before we prove our main result, we state two key propositions (we defer the proofs of these propositions to \S\ref{sec:proofs}). First, Proposition~\ref{prop: Probability Lemma} shows that if, for a  vertex $v \in V_k$, there exists a \textit{large} subset of colors $L^*\subseteq L(v)$ such that the expected number of problematic edges in $E(v)$ with respect to colors in $L^*$ is \textit{small}, then $\P[\mathcal{B}_v]$ is exponentially small in $\Delta$.

\begin{prop}\label{prop: Probability Lemma}
Let $\gamma, a \in (0, 1)$ be fixed, let $\Delta \in \N$ be sufficiently large in terms of $\gamma$ and $a$, and let
\(
q\coloneqq\left(\frac{(k-1)(a+2\gamma)}{1-p} \frac{\Delta}{\log\Delta}\right)^{\!\frac1{k-1}}
\).
Suppose that $|L(u)| \geq q$ for each $u \in V(H)$.
Additionally, for a vertex $v\in V_k$, suppose that there exists a set $L^*\subseteq L(v)$ of size at least $\gamma q$ such that the average value $\rho(v,c)$ for $c\in L^{*}$ satisfies
\[
\frac{1}{|L^{*}|}\sum_{c\in L^{*}}\rho(v,c)\ \le\ \Delta\left(\frac{a+\gamma}{q^{k-1}}\right).
\]
Then 
$\P[\mathcal{B}_v] \leq \exp\left(-\Delta^{\gamma/(3k)}\right)$.
\end{prop}

Next, Proposition~\ref{prop: Existence Lemma} confirms the existence of a subset $L^{*}\subseteq L(v)$ for each $v \in V_k$ satisfying the conditions of Proposition~\ref{prop: Probability Lemma} for $a = 4/5$.

\begin{prop}\label{prop: Existence Lemma}
Let $a = \frac45$, and let $\gamma,\Delta,q$ be as in Proposition~\ref{prop: Probability Lemma}.
Suppose that $|L(u)| \geq q$ for each $u \in V(H)$.
For each vertex $v \in V_k$, there exists a subset $L^{*} \subseteq L(v)$ with $|L^{*}| \ge \gamma q$ such that
\[
  \frac{1}{|L^{*}|}\sum_{c \in L^{*}} \rho(v,c)
  \;\le\; 
  \Delta\left(\frac{a + \gamma}{q^{\,k-1}}\right).
\]
\end{prop}

With the above propositions in hand, we prove our main result (Theorem~\ref{theorem: constant inside}).

\begin{proof}[Proof of Theorem~\ref{theorem: constant inside}]
    Assign each vertex $v \in V(H)$ a list $L(v)$ of $q = \left(\frac{4}{5} (k - 1 + \eps) \frac{\Delta}{\log\Delta}\right)^{\frac{1}{k - 1}}$ colors represented as nonnegative integers. 
    Sample a partial $L$-coloring $\phi$ of $H$ which colors vertices in parts \(V_1,\dots,V_{k-1}\) according to the distribution in \eqref{eq: prob dist}.
    Note the following for $a = 4/5$, $\gamma = \eps/(4k)$, and $p$ as defined in \eqref{eq:  p definition}:
    \[
    \left(\frac{4}{5}(k - 1 + \eps)\frac{\Delta}{\log\Delta}\right)^{\frac{1}{k - 1}} \geq \left(\frac{(k - 1)(a + 2\gamma)}{1-p}\frac{\Delta}{\log\Delta}\right)^{\frac{1}{k - 1}},\]
    for $\Delta$ sufficiently large.
    In particular, we may apply Propositions~\ref{prop: Probability Lemma} and \ref{prop: Existence Lemma} with $a=4/5$, $\gamma = \eps/(4k)$, and $p$ as defined in \eqref{eq:  p definition}.
    It follows that 
        \begin{align*}
        \P[\mathcal{B}_v]\leq \exp\left(-\Delta^{\gamma/(3k)}\right) = \exp\left(-\Delta^{\epsilon/(12k^2)}\right) \qquad \text{for each } v \in V_k,
        \end{align*}
    where $\mathcal{B}_v$ is as defined in \eqref{Bad event definiton}.

    Note that the event \(\mathcal{B}_w\) is determined by the random colors of the vertices in \(N(w)\), and hence \(\mathcal{B}_w\) is
    mutually independent of all \(\mathcal{B}_{w'}\) with \(N(w)\cap N(w')=\emptyset\).
    Moreover,
    \[
    \max_{w\in V_k}\left|\{\,w' \in V_k\,:\, w'\neq w,\, N(w)\cap N(w')\neq\emptyset\,\}\right|
    \;<\; \max_{w \in V_k}|N(w)|\cdot \Delta \;\le\; (k-1)\Delta^{2},
    \]
    and so we may apply the \hyperref[theorem: Lovasz Local Lemma]{Lovasz Local Lemma} (Theorem~\ref{theorem: Lovasz Local Lemma}) with
    \(p_{LLL}\coloneqq \exp\left(-\Delta^{\epsilon/(12k^2)}\right)\), \(d_{LLL}\coloneqq(k-1)\Delta^2-1\), and \(\{A_1, \dots A_n\} = \{\mathcal{B}_w : w \in V_k\}\), yielding
    \[
    e\,p_{LLL}\,(d_{LLL}+1)
    = e\,\exp\!\left(-\Delta^{\epsilon/(12k^2)}\right)(k-1)\Delta^2 < 1,
    \]
    for \(\Delta\) sufficiently large. Thus, with positive probability, none of the events
    \(\{\mathcal{B}_w : w \in V_k\}\) occur, completing the proof.
\end{proof}

\section{Proofs of Propositions}\label{sec:proofs}
In this section, we include the deferred proofs of the propositions used in the proof of Theorem~\ref{theorem: constant inside}. Specifically, we prove Propositions~\ref{prop: Probability Lemma} and~\ref{prop: Existence Lemma}, which together imply that for every $v \in V_k$, the probability that $v$ is blocked under the random coloring procedure of \eqref{eq: prob dist} is small. We further split this section into two subsections, devoted to the proofs of each of the propositions.

\subsection{Probability bound}
In this section, we will prove Proposition~\ref{prop: Probability Lemma}, which states that if, for a vertex $v \in V_k$,  there exists $L^* \subseteq L(v)$ such that the expected number of edges that are problematic with respect to colors in $ L^{*} $ is sufficiently small, then the probability that $v$ is blocked is also small. For the reader's convenience, we restate the proposition here.

\begin{prop*}[Restatement of Proposition \ref{prop: Probability Lemma}]
Let $\gamma, a \in (0, 1)$ be fixed, let $\Delta \in \N$ be sufficiently large in terms of $\gamma$ and $a$, and let
\(
q\coloneqq\left(\frac{(k-1)(a+2\gamma)}{1-p} \frac{\Delta}{\log\Delta}\right)^{\!\frac1{k-1}}
\).
Suppose that $|L(u)| \geq q$ for each $u \in V(H)$.
Additionally, for a vertex $v\in V_k$, suppose that there exists a set $L^*\subseteq L(v)$ of size at least $\gamma q$ such that the average value $\rho(v,c)$ for $c\in L^{*}$ satisfies
\[
\frac{1}{|L^{*}|}\sum_{c\in L^{*}}\rho(v,c)\ \le\ \Delta\left(\frac{a+\gamma}{q^{k-1}}\right).
\]
Then 
$\P[\mathcal{B}_v] \leq \exp\left(-\Delta^{\gamma/(3k)}\right)$.
\end{prop*}

\begin{proof}
We begin by showing that
\begin{equation}\label{eq:inner}
   \P[\mathcal{B}_v]
   \;\le\;
   \exp\left(
       -\sum_{c\in L(v)}
          \exp \left(-\frac{\rho(v,c)}{1-p}\right)
   \right).
\end{equation}
For each color $c\in L(v)$, define the indicator variable 
\[
X_c \coloneqq \mathbf{1}\{\exists\, e \in E(v) \,:\, \phi_{e}=c\},
\]
i.e., $X_c=1$  if and only if $c$ is blocked at $v$, and $X_c=0$ otherwise. Set
\[
X \coloneqq \prod_{c\in L(v)} X_c,
\]
so that $X=1$ precisely when every color in $L(v)$ is blocked at $v$, i.e., the event $\mathcal{B}_v$ occurs; thus, it suffices to provide an upper bound for \(\P[X=1]\). 
To this end, we provide an upper bound for $\P[X_c = 1]$.
\begin{claim}\label{claim4.1.1}
    \(\P[X_c=1]\leq 1-\prod_{e \in E(v)}(1-\P_e(c))\)
\end{claim}
\begin{claimproof}
    We show this by providing a lower bound for the value of $\P[X_c=0]$ using \hyperref[theo:harris]{Harris's Inequality} (Theorem \ref{theo:harris}). To this end, we let $X_e$ be the indicator for an edge $e \in E(v)$ being problematic with respect to $c$, i.e.,
    \[X_e \;=\; \mathbf 1_{\{\phi_e=c\}}.\]
    Note that \(X_c=0 \) if and only if \(X_e=0\) for all \(e \in E(v)\).
    Let \[\Gamma_c\coloneqq \{u \in V(H)\setminus V_k:c \in L(u)\}\] be the set of vertices in $V(H) \setminus V_k$ which can be colored with the color $c$, and let $S\subseteq \Gamma_c$ be the random subset such that $u \in S$ if and only if $\phi(u)=c$. For each $u \in \Gamma_c$, $u$ is included in $S$ independently with probability $\P_{u,c}$. We also note that for each $e \in E(v)$, $X_e=1$ if and only if $u \in S$ for each $u \in e\setminus \{v\}$. For each $e \in E(v)$, we define a family $\mathcal A_e$ of subsets of $\Gamma_c$ as follows:
    \[\mathcal{A}_e\coloneqq\{S'\subseteq \Gamma_c\,:\,X_e=0 \text{ when } S = S'\}.\]
    Equivalently, $\mathcal{A}_e$ consists of $S'\subseteq \Gamma_c$ such that there exists $u\in e\setminus \{v\}$ that is not in $S'$. Since $S_1\in \mathcal{A}_e$ implies that there is some vertex $u \in e\setminus \{v\}$ which is not in $S_1$, every subset $S_2 \subseteq S_1$ must also satisfy $u \notin S_2$ and so $S_2 \in \mathcal{A}_e$ as well. In particular, $\mathcal{A}_e$ is a decreasing family of subsets of $\Gamma_c$. Hence, by \hyperref[theo:harris]{Harris's Inequality} (Theorem \ref{theo:harris}), we have
    \[
      \P[X_c = 0]
      \;=\;
      \P\left[S \in\bigcap_{e \in E(v)  } \mathcal{A}_e\right]
      \;\ge\;
      \prod_{e \in E(v) }\P[S \in\mathcal{A}_e]
      \;=\;
      \prod_{e \in E(v) }\left(1-\P_e(c)\right).
    \] It follows that \[
      \P[X_c = 1]
      \;=\;
      1 - \P[X_c = 0]
      \;\le\;
      1 - \prod_{e \in E(v)}\left(1-\P_e(c)\right),
    \]
as desired.
\end{claimproof}

Applying the inequality  $1-x\geq \exp(\frac{-x}{1-x})>\exp(\frac{-x}{1-p})$ for $x<p$ along with Claim \ref{claim4.1.1}, we have
\begin{align}\label{eq:inner0}
    \P[ X_c=1 ]<1-\exp\left(\frac{-\sum_{e \in E(v)} \P_{e}(c)}{1-p}\right)=1-\exp\left(-\frac{\rho(v,c)}{1-p}\right),
\end{align}
where we use the definition of $\rho(v, c)$ from \eqref{eq: rho}.

Next, we claim that the events $\{X_c=1\}_{c \in L(v)}$ are negatively correlated. The argument is identical to \cite[Claim 4.1.2]{dhawan2025list} and so we omit the details.
Intuitively, conditioning on $X_c=1$ implies that there exists an edge $e \in E(v)$ that is problematic with respect to color $c$. The presence of such an edge can only reduce the probability that there exists another edge $f \in E(v)$ that is problematic with respect to some other color $c'$. Consequently, $\{X_c=1\}$ and $\{X_{c'}=1\}$ are negatively correlated. The argument of \cite[Claim 4.1.2]{dhawan2025list} carefully applies this observation to prove the desired result. We note that the argument relies solely on the fact that colors are chosen independently for each vertex in $V \setminus V_k$, i.e., the probability distribution is unimportant and so the argument applies in our setting as well.

With this in hand, we have
\[
  \P \left[\forall c\in L(v),\;X_c=1\right]
  \;\le\;
  \prod_{c\in L(v)}\P[X_c=1].
\]
Combining the above inequality with \eqref{eq:inner0}, we obtain
\[
 \P[\mathcal{B}_v]
  \;\le\;
  \prod_{c\in L(v)}
  \left(1-\exp\!\left(-\frac{\rho(v,c)}{1-p}\right)\right).
\]
Using \(1-x\le e^{-x}\), we have
\[
  \prod_{c\in L(v)}
  \left(1-\exp\left({-\frac{\rho(v,c)}{1-p}}\right)\right)
  \;\le\;
  \exp\left(-\sum_{c\in L(v)}\exp\left({-\frac{\rho(v,c)}{1-p}}\right)\right),
\]
which completes the proof of \eqref{eq:inner}.
From the statement of Proposition \ref{prop: Probability Lemma}, since $L^* \subseteq L(v)$, we conclude that
\begin{equation}\label{eq:outer}
  \P[\mathcal{B}_v]
  \;\le\;
  \exp\left(-\sum_{c\in L^{\!*}}\exp\left({-\frac{\rho(v,c)}{1-p}}\right)\right).
\end{equation}
Since \(f(x)=e^{-x}\) is convex, we apply \hyperref[theo:Jensen]{Jensen's inequality} (Theorem \ref{theo:Jensen}) to obtain
\[\sum_{c\in L^{\!*}}f\left(\frac{\rho(v,c)}{1-p}\right) 
\geq |L^*|f\left(\frac1{|L^*|}\sum_{c\in L^{\!*}}\frac{\rho(v,c)}{1-p}\right) 
\geq \gamma q f\left(\frac1{|L^*|} \sum_{c\in L^{\!*}}\frac{\rho(v,c)}{1-p}\right).\]
With the hypothesis \(\frac1{|L^*|}\sum_{c\in L^{\!*}}\rho(v,c)\le\Delta\frac{a + \gamma}{q^{k-1}}\), we deduce that
\[
  \sum_{c\in L^{\!*}}\exp \left({-\frac{\rho(v,c)}{1-p}}\right)
  \;\ge\;
  \gamma q\,
  \exp\!\left(-\frac{(a + \gamma)\Delta/q^{k-1}}{1-p}\right).
\]
Plugging the above back into \eqref{eq:outer}, we have
\begin{equation}\label{eq:outer2}
 \P[\mathcal{B}_v]
  \;\le\;
  \exp\left(
    -\gamma q\,
    \exp \left(-\frac{(a +\gamma)\Delta/q^{k-1}}{1-p}\right)
  \right).
\end{equation}
Plugging in $q =  \left ( \frac{(k - 1)(a + 2\gamma)}{1-p} \frac{\Delta}{\log \Delta} \right )^{\frac{1}{k-1}}$, we obtain
\[
\label{eq:cond}
\frac{(a + \gamma)\Delta}{(1-p)q^{k-1}}
\;=
\frac{(a + \gamma)\log \Delta}{(k - 1)(a + 2\gamma)}.
\]
It follows that
\[
\exp\!\left(-\frac{(a + \gamma)\Delta}{(1-p)q^{k-1}}\right)
\;=
\exp\!\left(-\frac{(a + \gamma)\log \Delta}{(k-1)(a + 2\gamma)}\right)
=
\Delta^{-\frac{a + \gamma}{(k-1)(a + 2\gamma)}}.
\]
Substituting this into \eqref{eq:outer2} gives
\begin{align*}
   \P[\mathcal{B}_v]
&\le\;
\exp\!\left(
  -\gamma q \,\Delta^{-\frac{a + \gamma}{(k-1)(a + 2\gamma)}}
\right), \\
& = \exp \left( -\gamma\left(\frac{(k-1)(a + 2\gamma)}{(1-p)\log \Delta}\right)^{\!1/(k-1)}
\Delta^{\frac{1}{k-1} -\frac{a + \gamma}{(k-1)(a + 2\gamma)}} \right) \\
& \le \exp \left( -\gamma\left(\frac{(k-1) \,(a + 2\gamma)}{\log \Delta}\right)^{\!1/(k-1)}
\Delta^{\frac{\gamma}{(k-1)(a + 2\gamma)}}\right).
\end{align*}
Let
\(
C\coloneqq \gamma\left((k-1)(a + 2\gamma)\right)^{1/(k-1)}
\) be a positive constant depending only on \(a,k,\gamma\). 
We have
\[
\P[\mathcal{B}_v]
\le\ \exp \left( - C 
\frac{\Delta^{\frac{\gamma}{(k-1)(a+2\gamma)}}}{(\log \Delta)^{\frac{1}{k-1}}} 
\right) \leq
\exp\left(-\Delta^{\gamma/(3k)}\right),
\]
for $\Delta$ sufficiently large, completing the proof.
\end{proof}

\subsection{Existence of $L^*$}

In this section, we will prove Proposition~\ref{prop: Existence Lemma}, which guarantees that for every vertex $v \in V_k$ there exists a subset 
$L^{*}\subseteq L(v)$ satisfying the hypotheses of Proposition~\ref{prop: Existence Lemma} for $a = 4/5$. For the reader's convenience, we restate the proposition here.

\begin{prop*}[{Restatement of Proposition~\ref{prop: Existence Lemma}}]
Let $a = \frac45$, and let $\gamma,\Delta,q$ be as in Proposition~\ref{prop: Probability Lemma}.
Suppose that $|L(u)| \geq q$ for each $u \in V(H)$.
For each vertex $v \in V_k$, there exists a subset $L^{*} \subseteq L(v)$ with $|L^{*}| \ge \gamma q$ such that
\[
  \frac{1}{|L^{*}|}\sum_{c \in L^{*}} \rho(v,c)
  \;\le\; 
  \Delta\left(\frac{a + \gamma}{q^{\,k-1}}\right).
\]
\end{prop*}

\begin{proof}

For an edge $e \in E$, let the set of colors common to all vertices of $e$ be denoted by
\[
    L_e \coloneqq \bigcap_{u\in e} L(u).
\]
For convenience, we introduce the shorthand
\[
    i_{u,c} \coloneqq I(L(u), c),
    \qquad
    \text{and}
    \qquad
    i_{e,c} \coloneqq I(L_e, c).
\]
Clearly, for each $u \in e$, as $L_e \subseteq L(u)$, 
\[
    i_{e,c} \le i_{u,c} \qquad \text{for all } c \in L_e.
\]

Now fix a vertex $v\in V_k$. For each edge $e\in E(v)$, write its vertices as
\[
    e=\{u_1^e,u_2^e,\dots,u_k^e\}, \qquad \text{where } u_i^e \in V_i \text{ for each } i.
\] 
Observe that $u_k^e = v$.
Consider the probability that an edge $e \in E(v)$ is problematic with respect to the color \(c\). We have
\begin{align}\label{eq: problematic prob ub}
    \P_e(c) = \prod_{u\in e} \P_{u,c} = \left\{\begin{array}{cc}
        \frac{\frac{8}{5}}{q^{k-1}\left(1-\frac{3}{5q}\right)}\left(1 - \frac{3}{4q}i_{u_1^e,c}\right) & \text{if } c \in L_e; \\[12pt]
        0 & \text{otherwise}.
    \end{array}\right.
\end{align}

To assist with the remainder of the proof, we define the average size of \(|L_e|\) across all edges in $E(v)$ as \[
z \;=\; \frac{1}{\Delta}\sum_{e\in E(v)}|L_e|.
\]
In other words, $z$ represents the size of $L_e$ for the average edge $e \in E(v)$. Note that any $e \in E(v)$ can only be problematic with respect to colors in $L_e$. We split the analysis into two cases based on the value of $z$ over its full range $[0,q]$.

First, we consider the case that $z$ is \textit{small}. In this case, the average edge in $E(v)$ can be problematic with respect to only a few colors, and so the probability that it is problematic with respect to some color is \textit{small}. As a result, the expected number of edges in $E(v)$ which are problematic with respect to colors in $L(v)$ is sufficiently small. In this regime, we choose $L^*=L(v)$ and prove that our choice of $L^*$ satisfies the desired condition.

\begin{claim}\label{claim: low z}
Suppose $z \le \frac{2q}{3}$. 
Then for $L^* = L(v)$, we have
\[
    \frac{1}{|L^*|}\sum_{c \in L^*}\rho(v,c) 
    \leq \Delta\left(\frac{4/5 + \gamma}{q^{k-1}}\right).
\]
\end{claim}
\begin{claimproof}
Recall that for each edge \(e \in E(v)\), its vertices are written as \(\{u_1^e, \dots , u_k^e \}\) such that \(u_i^e \in V_i\) for each $i$, so that $u^e_k = v$. We have 
\[
\sum_{c\in L_e} i_{u_1^e,c} \ge\sum_{c\in L_e}i_{e,c} = \frac{|L_e|(|L_e|+1)}{2}
> \frac{|L_e|^2}{2}.
\]
Since $f(x) = x^2$ is convex, by \hyperref[theo:Jensen]{Jensen's Inequality} (Theorem \ref{theo:Jensen}) we have
\[\sum_{e\in E(v)}{|L_e|^2} \ge \Delta\left(\sum_{e\in E(v)}\frac{|L_e|}{\Delta}\right)^2 = \Delta\,z^2.\]
We apply \eqref{eq: problematic prob ub} to obtain
\begin{align*}
    \frac{1}{|L^*|}\sum_{c \in L^*}\rho(v,c) 
    &= \frac{1}{|L(v)|} \sum_{c\in L(v)} \sum_{e \in E(v)} \P_e(c), \\
    &= \frac{1}{q} \sum_{e \in E(v)} \sum_{c\in L_e}\frac{\frac{8}{5}}{q^{k-1}\left(1-\frac{3}{5q}\right)}\left(1 - \frac{3}{4q}i_{u_1^e,c}\right) ,\\
    &=\frac{\frac{8}{5}}{q^{k-1}\left(1-\frac{3}{5q}\right)}\cdot\frac1q \sum_{e \in E(v)} \sum_{c\in L_e} \left(1 - \frac{3}{4q}i_{u_1^e,c}\right) ,\\
    &= \frac{\frac{8}{5}}{q^{k-1}\left(1-\frac{3}{5q}\right)}\cdot\frac1q \left( z \Delta - \frac{3}{4q} \sum_{e \in E(v)} \sum_{c \in L_e} i_{u_1^e, c}\right), \\
    &< \frac{\frac{8}{5}}{q^{k-1}\left(1-\frac{3}{5q}\right)} \left( \frac{z \Delta}{q} - \frac{3}{4q^2} \sum_{e\in E(v)}\frac{|L_e|^2}{2}\right), \\
    &\leq \frac{\frac{8}{5}}{q^{k-1}\left(1-\frac{3}{5q}\right)} \left(\frac{z \Delta}{q} - \frac{3\Delta\,z^2}{8q^2} \right).
\end{align*}
This function is increasing in the interval $z \in [0,\frac{2q}{3}]$, so it is maximized at $z = \frac{2q}{3}$.
It follows that
\[
    \frac{1}{|L^*|}\sum_{c \in L^*}\rho(v,c) 
    \leq  \frac{\frac{8}{5}}{q^{k-1}\left(1-\frac{3}{5q}\right)} \left( \frac{2\Delta}{3} - \frac{\Delta}{6}\right) 
    =  \frac{\frac{8}{5}}{q^{k-1}\left(1-\frac{3}{5q}\right)}\ \frac\Delta 2 \leq \Delta\left(\frac{4/5 + \gamma}{q^{k-1}}\right),
\]
for $\Delta$ sufficiently large.
\end{claimproof}

We now tackle the case where $z > 2q/3$, i.e., $|L_e|$ is \textit{large} on average. Before giving the formal argument, we briefly explain the intuition. Consider an average edge $e \in E(v)$. Since $| L_e |$ is large and since $L_e \subseteq L(u_1^e) \cap L(v)$, it follows that
$\lvert L(u_1^e) \cap L(v) \rvert \ge \lvert L_e \rvert$ is also large. A large intersection $L(u_1^e) \cap L(v)$ forces many of the high-index colors of $L(v)$ to also appear with high index in $L(u_1^e)$. In particular, these colors have a small probability of being assigned to $u_1^e$, implying that $e$ has a small probability of being problematic with respect to high-index colors of $L(v)$. 
Consequently, when $z > 2q/3$, the expected number of problematic edges with respect to high-index colors of $L(v)$ remains \textit{small}. To formalize this idea, we choose the set of $\gamma q$ highest index colors in $L(v)$ to be $L^*$ and prove that it satisfies the desired condition. 

\begin{claim}\label{claim: high z}
    Suppose $z > \frac{2q}{3}$. 
    Then for $L^*  \coloneqq \{c \in L(v) \,:\, i_{v, c} \ge (1 -\gamma)q\}$, we have
    \[
        \frac{1}{|L^*|}\sum_{c \in L^*}\rho(v,c) 
        \leq \Delta\left(\frac{4/5 + \gamma}{q^{k-1}}\right).
    \]
\end{claim}

\begin{claimproof}
Recall that for each edge \(e \in E(v)\), its vertices are written as \(\{u_1^e, \dots , u_k^e \}\) such that \(u_i^e \in V_i\) for each $i$, so that $u_k^e = v$. Note the following:
\begin{align*}
    \sum_{c\in L^*}i_{v,c} = \frac{|L^*|(2q - |L^*| + 1)}{2} = \frac{2\gamma q^2 - (\gamma q)^2 + \gamma q}{2} .
\end{align*} 
For any $c \in L_e$, since \(L_e \subseteq L(v)\), the set of colors in \(L_e\) that are larger than
\(c\) is contained in the set of colors in \(L(v)\) that are larger than \(c\).
Consequently, we have
\[
  |L_e| - i_{e,c} \le |L(v)| - i_{v,c},\qquad \forall c \in L_e
\]
and hence
\[
  i_{e,c} \ge |L_e| - (q - i_{v,c}),\qquad \forall c \in L_e.
\]
To simplify computations, we let $I(L, c) = 4q/3$ for a list-color pair $(L, c)$ satisfying $c \notin L$ (note that if $L = L(u)$ for some $u\in V_1$, this choice sets $\P_{u,c} = 0$). 
In particular, we let $i_{e, c} = 4q/3$ for $c \notin L_e$.
With this in mind and noting that $i_{v, c} \leq q$, we may extend the above inequality to 
\[
  i_{e,c} \ge |L_e| - (q - i_{v,c}),\qquad \forall c \in L(v).
\]
We obtain
\begin{align*}
    \sum_{c\in L^*}i_{e,c}  
    \geq \sum_{c\in L^*}(|L_e| - (q - i_{v,c})) 
    &= |L_e| |L^*| - q |L^*| + \sum_{c \in L^*} i_{v,c} \\
    &= \gamma q |L_e| - \gamma q^2 + \frac{2\gamma q^2 - (\gamma q)^2 + \gamma q}{2} \\
    &\ge \frac{\gamma q \left(2|L_e| - \gamma q \right)}{2}.
\end{align*}
Recall that $i_{u_1^e, c} \geq i_{e, c}$. Applying \eqref{eq: problematic prob ub}, we have
\begin{align*}
    \frac{1}{|L^*|}\sum_{c \in L^*}\rho(v,c)
    &= \frac{1}{|L^*|}\sum_{c\in L^*}\sum_{e\in E(v)}\P_e(c)\\
    &\leq \frac{1}{|L^*|}\sum_{c\in L^*}\sum_{e\in E(v)}
    \frac{\frac{8}{5}}{q^{k-1}\left(1-\frac{3}{5q}\right)}
    \left(1-\frac{3}{4q}i_{u_1^e,c}\right)\\
    &= \frac{\frac{8}{5}}{q^{k-1}\left(1-\frac{3}{5q}\right)}\cdot \frac{1}{|L^*|}
    \sum_{c\in L^*}\sum_{e\in E(v)}
    \left(1-\frac{3}{4q}i_{u_1^e,c}\right)\\
    &= \frac{\frac{8}{5}}{q^{k-1}\left(1-\frac{3}{5q}\right)}\cdot \frac{1}{|L^*|}
    \left(\Delta|L^*|-\sum_{c\in L^*}\sum_{e\in E(v)}\frac{3}{4q}i_{u_1^e,c}\right)\\
    &\le \frac{\frac{8}{5}}{q^{k-1}\left(1-\frac{3}{5q}\right)}\cdot \frac{1}{\gamma q}
    \left(\Delta \gamma q -\sum_{e\in E(v)}\sum_{c\in L^*}\frac{3}{4q}i_{e,c}\right)\\
    &\le \frac{\frac{8}{5}}{q^{k-1}\left(1-\frac{3}{5q}\right)}\cdot \frac{1}{\gamma q}
    \left(\Delta\gamma q-\frac{3}{4q}\sum_{e\in E(v)} \frac{\gamma q \left(2|L_e| - \gamma q \right)}{2}\right)\\
    &= \frac{\frac{8}{5}}{q^{k-1}\left(1-\frac{3}{5q}\right)}
    \left(\Delta -\frac{3}{8q}\,\sum_{e\in E(v)}\left(2|L_e|-\gamma q\right)\right)\\
    &= \frac{\frac{8}{5}}{q^{k-1}\left(1-\frac{3}{5q}\right)}
    \left(\Delta-\frac{3}{8q}\left(2\Delta z-\gamma q\,\Delta\right)\right).
\end{align*}

Applying the bound $z > \frac{2q}{3}$, we obtain
\begin{align*}
\frac{1}{|L^*|}\sum_{c \in L^*}\rho(v,c) 
&< \frac{\frac{8}{5}}{q^{k-1}\left(1-\frac{3}{5q}\right)}\left(\Delta - \frac{3}{8q}\left(\frac{4\Delta q}{3} - \gamma q\Delta\right)\right) \\
&=\frac{\frac{8}{5}}{q^{k-1}\left(1-\frac{3}{5q}\right)} \left( \Delta - \frac{\Delta}{2} + \frac{3\gamma\Delta}{8} \right) \\
&= \frac{\frac{8}{5}}{q^{k-1}\left(1-\frac{3}{5q}\right)}\left(\frac{1}{2} + \frac{3\gamma}{8}\right)\Delta\\
&\le \frac{\frac{8}{5}}{q^{k-1}}\left(\frac{1+ \gamma}{2} \right)\Delta \\ 
&\leq \Delta\left(\frac{4/5 + \gamma}{q^{k-1}}\right),
\end{align*}
for $\Delta$ sufficiently large, completing the proof.
\end{claimproof}

By Claims~\ref{claim: low z} and \ref{claim: high z}, in the interval $z =[0,q]$ (which is the entire feasible range of $z$), at least one of $L^* = L(v)$ or $L^* = \{c \in L(v) \,:\, i_{v, c} \geq (1-\gamma)q\}$
satisfies
\[
  \frac{1}{|L^*|}\sum_{c \in L^*}\rho(v,c) \leq \Delta\frac{4/5 + \gamma}{q^{k-1}},
\]
completing the proof of Proposition~\ref{prop: Existence Lemma}.
\end{proof}

\printbibliography

@book{AlonSpencer,
	author = {N. Alon and J.H. Spencer},
	title = {The Probabilistic Method},
	date = {2000},
	edition = {2},
	publisher = {John Wiley {\&} Sons},
}

@article{Harris,
	author = {T.E. Harris},
	title = {A lower bound for the critical probability in a certain percolation process},
	journaltitle = {Math. Proc. Cambridge Phil. Soc.},
	date = {1960},
	volume = {56},
	number = {1},
	pages = {13--20},
}

@article{Kleitman,
	author = {D.J. Kleitman},
	title = {Families of non-disjoint subsets},
	journaltitle = {J. Combin. Theory},
	date = {1966},
	volume = {1},
	number = {1},
	pages = {153--155},
}

@inproceedings{brooks1941colouring,
  title={On colouring the nodes of a network},
  author={R. L. Brooks},
  booktitle={Mathematical Proceedings of the Cambridge Philosophical Society},
  volume={37},
  number={2},
  pages={194--197},
  year={1941},
  organization={Cambridge University Press}
}

@article{reed1999strengthening,
  title={A strengthening of Brooks' theorem},
  author={B. Reed},
  journal={Journal of Combinatorial Theory, Series B},
  volume={76},
  number={2},
  pages={136--149},
  year={1999},
  publisher={Academic Press}
}

@ARTICLE{AKSConjecture,
    AUTHOR = "N. Alon and M. Krivelevich and B. Sudakov",
    TITLE = "{Coloring graphs with sparse neighborhoods}",
    JOURNAL = "J. Combin. Theory",
    series = {B},
    YEAR = "1999",
    volume = {77},
    pages = {73--82},
}

@article{BollobasIndependence,
    author = {B. Bollob{\'{a}}s},
    title = {The independence ratio of regular graphs},
    journaltitle = {Proc. Amer. Math. Soc.},
    date = {1981},
    volume = {83},
    number = {2},
    pages = {433--436},
}

@report{Joh_triangle,
    author = {A. Johansson},
    title = {Asymptotic choice number for triangle free graphs},
    type = {Technical Report 91--95},
    institution = {DIMACS},
    date = {1996},
}

@unpublished{DKPS,
    author = {E. Davies and R.J. Kang and F. Pirot and J.-S. Sereni},
    title = {Graph structure via local occupancy},
    howpublished = {\url{https://arxiv.org/abs/2003.14361} (preprint)},
    date = {2020},
}

@article{anderson2023colouring,
  title={Colouring graphs with forbidden bipartite subgraphs},
  author={J. Anderson and A. Bernshteyn and A. Dhawan},
  journal={Combinatorics, Probability and Computing},
  volume={32},
  number={1},
  pages={45--67},
  year={2023},
  publisher={Cambridge University Press}
}

@article{anderson2025coloring,
  title={Coloring graphs with forbidden almost bipartite subgraphs},
  author={Anderson, James and Bernshteyn, Anton and Dhawan, Abhishek},
  journal={Random Structures \& Algorithms},
  volume={66},
  number={4},
  pages={e70012},
  year={2025},
  publisher={Wiley Online Library}
}

@article{bernshteyn2019johansson,
  title={The Johansson-Molloy theorem for DP-coloring},
  author={A. Bernshteyn},
  journal={Random Structures \& Algorithms},
  volume={54},
  number={4},
  pages={653--664},
  year={2019},
  publisher={Wiley Online Library}
}

@article{frieze2013coloring,
  title={Coloring simple hypergraphs},
  author={A. Frieze and D. Mubayi},
  journal={Journal of Combinatorial Theory, Series B},
  volume={103},
  number={6},
  pages={767--794},
  year={2013},
  publisher={Elsevier}
}

@article{cooper2016coloring,
  title={Coloring sparse hypergraphs},
  author={J. Cooper and D. Mubayi},
  journal={SIAM Journal on Discrete Mathematics},
  volume={30},
  number={2},
  pages={1165--1180},
  year={2016},
  publisher={SIAM}
}

@unpublished{li2022chromatic,
  title={The chromatic number of triangle-free hypergraphs},
  author={L. Li and L. Postle},
howpublished = {\url{https://arxiv.org/abs/2202.02839} (preprint)},
  year={2022}
}

@article{cambie2022independent,
  title={Independent transversals in bipartite correspondence-covers},
  author={S. Cambie and R.J. Kang},
  journal={Canadian Mathematical Bulletin},
  volume={65},
  number={4},
  pages={882--894},
  year={2022},
  publisher={Canadian Mathematical Society}
}

@article{alon2021asymmetric,
  title={Asymmetric list sizes in bipartite graphs},
  author={N. Alon and S. Cambie and R.J. Kang},
  journal={Annals of Combinatorics},
  volume={25},
  number={4},
  pages={913--933},
  year={2021},
  publisher={Springer}
}

@article{alon1998choice,
  title={The choice number of random bipartite graphs},
  author={N. Alon and M. Krivelevich},
  journal={Annals of Combinatorics},
  volume={2},
  number={4},
  pages={291--297},
  year={1998},
  publisher={Springer}
}

@article{meroueh2019list,
  title={List colorings of multipartite hypergraphs},
  author={A. M{\'e}roueh and A. Thomason},
  journal={Random Structures \& Algorithms},
  volume={55},
  number={4},
  pages={950--979},
  year={2019},
  publisher={Wiley Online Library}
}

@article{haxell2010list,
  title={List coloring hypergraphs},
  author={P. Haxell and J. Verstraete},
  journal={the electronic journal of combinatorics},
  pages={R129--R129},
  year={2010}
}

@article{erdos1979choosability,
  title={Choosability in graphs},
  author={P. Erd\H{o}s and A.L. Rubin and H. Taylor},
  journal={Congr. Numer},
  volume={26},
  number={4},
  pages={125--157},
  year={1979}
}

@unpublished{bradshaw2024bipartite,
  title={Bipartite graphs are $(\frac{4}{5}-\varepsilon)\frac{\Delta}{\log\Delta}$-choosable},
  author={Bradshaw, Peter and Mohar, Bojan and Stacho, Ladislav},
  howpublished = {\url{https://arxiv.org/abs/2409.01513} (preprint)},
  year={2024}
}

@article{dhawan2025list,
  title={List colorings of $ k $-partite $ k $-graphs},
  author={Dhawan, Abhishek},
  journal={The Electronic Journal of Combinatorics},
  volume={32},
  issue ={2},
  pages = {P2.16},
  doi = {https://doi.org/10.37236/12657},
  year={2025}
}

@article{Achlioptas,
	author = {D. Achlioptas and A. Coja-Oghlan},
	title = {Algorithmic barriers from phase transitions},
	journaltitle = {IEEE Symposium on Foundations of Computer Science (FOCS)},
	date = {2008},
	pages = {793--802},
	addendum = {Full version: \url{https://arxiv.org/abs/0803.2122}},
}

@article{Zdeborova,
	author = {L. Zdeborov{\'{a}} and F. Krz{\k{a}}ka{\l}a},
	title = {Phase transitions in the coloring of random graphs},
	journaltitle = {Phys. Rev. E},
	date = {2007},
	volume = {76},
	pages = {031131},
}

@unpublished{martinsson2021simplified,
  title={A simplified proof of the Johansson-Molloy Theorem using the Rosenfeld counting method},
  author={Martinsson, Anders},
  howpublished = {\url{https://arxiv.org/abs/2111.06214} (preprint)},
  year={2021}
}

@unpublished{hurley2021first,
  title={A first moment proof of the Johansson-Molloy theorem},
  author={Hurley, EOIN and Pirot, FRAN{\c{C}}OIS},
  howpublished = {\url{https://arxiv.org/abs/2109.15215} (preprint)},
  year={2021}
}

@article{karp1976probabilistic,
  title={The probabilistic analysis of some combinatorial search algorithms.},
  author={Karp, Richard M.},
  year={1976}
}

@unpublished{dhawan2024low,
  title={The low-degree hardness of finding large independent sets in sparse random hypergraphs},
  author={Dhawan, Abhishek and Wang, Yuzhou},
  howpublished = {\url{https://arxiv.org/abs/2404.03842} (preprint)},
  year={2024}
}

@inproceedings{iliopoulos2021improved,
  title={Improved Bounds for Coloring Locally Sparse Hypergraphs},
  author={Iliopoulos, Fotis},
  booktitle={Approximation, Randomization, and Combinatorial Optimization. Algorithms and Techniques (APPROX/RANDOM 2021)},
  pages={1--39},
  year={2021},
  organization={Schloss Dagstuhl--Leibniz-Zentrum f{\"u}r Informatik}
}

@unpublished{bradshaw2025toward,
  title={Toward Vu's conjecture},
  author={Bradshaw, Peter and Dhawan, Abhishek and Methuku, Abhishek and Wigal, Michael C},
  howpublished = {\url{https://arxiv.org/abs/2508.16818} (preprint)},
  year={2025}
}

@article{bonamy2022bounding,
  title={Bounding $\chi$ by a fraction of $\Delta$ for graphs without large cliques},
  author={Bonamy, Marthe and Kelly, Tom and Nelson, Peter and Postle, Luke},
  journal={Journal of Combinatorial Theory, Series B},
  volume={157},
  pages={263--282},
  year={2022},
  publisher={Elsevier}
}

@unpublished{verstraete2024independent,
  title={Independent sets in hypergraphs},
  author={Verstraete, Jacques and Wilson, Chase},
  howpublished = {\url{https://arxiv.org/abs/2409.19908} (preprint)},
  year={2024}
}

@article{ajtai1982extremal,
  title={Extremal uncrowded hypergraphs},
  author={Ajtai, Mikl{\'o}s and Koml{\'o}s, J{\'a}nos and Pintz, Janos and Spencer, Joel and Szemer{\'e}di, Endre},
  journal={Journal of Combinatorial Theory, Series A},
  volume={32},
  number={3},
  pages={321--335},
  year={1982},
  publisher={Elsevier}
}

@book{hardy1952inequalities,
  title={Inequalities},
  author={Hardy, Godfrey Harold and Littlewood, John Edensor and P{\'o}lya, George},
  year={1952},
  publisher={Cambridge university press}
}

\end{document}